\documentclass{amsart}
\usepackage{amsmath,amsfonts,amssymb}
\usepackage{amsthm}
\usepackage{ifthen}
\usepackage{longtable}
\usepackage{array}
\usepackage{url}
\usepackage{multirow}
\usepackage{graphicx}

\newcommand{\Z}{\mathbb{Z}}
\newcommand{\Q}{\mathbb{Q}}
\newcommand{\R}{\mathbb{R}}
\newcommand{\C}{\mathbb{C}}
\newcommand{\N}{\mathbb{N}}

\newcommand{\ord}{\mathfrak{o}}

\renewcommand{\imath}{i}

\newtheorem{definition}{Definition}
\newtheorem{theorem}{Theorem}
\newtheorem{lemma}{Lemma}

\newtheorem{proposition}{Proposition}

\newtheorem{claim}{Claim}

\theoremstyle{remark}
\newtheorem{remark}{Remark}
\renewcommand{\Re}{\mathrm{Re}}
\renewcommand{\Im}{\mathrm{Im}}

\begin{document}
 \title[Distinct unit generated fields]{On distinct unit generated fields that are totally complex}
\subjclass[2010]{} \keywords{}

\author[D. Dombek]{Daniel Dombek}
\address{D. Dombek \newline
\indent Department of Applied Mathematics\newline
\indent Faculty of Information Technology, CTU in Prague\newline
\indent Th\'akurova 9\newline
\indent 160 00 Prague 6, Czech Republic}
\email{daniel.dombek\char'100fit.cvut.cz}

\author[Z. Mas\'akov\'a]{Zuzana Mas\'akov\'a}
\address{Z. Mas\'akov\'a \newline
\indent Department of Mathematics\newline
\indent Faculty of Nuclear Sciences and Physical Engineering, CTU in Prague\newline
\indent Trojanova 13\newline
\indent 120 00 Prague 2, Czech Republic}
\email{zuzana.masakova\char'100fjfi.cvut.cz}

\author[V. Ziegler]{Volker Ziegler}
\address{V. Ziegler\newline
\indent Johann Radon Institute for Computational and Applied Mathematics (RICAM)\newline
\indent Austrian Academy of Sciences\newline
\indent Altenbergerstr. 69\newline
\indent A-4040 Linz, Austria}
\email{volker.ziegler\char'100ricam.oeaw.ac.at}


\begin{abstract}
We consider the problem of characterizing all number fields $K$ such that all algebraic integers $\alpha\in K$ can be written as the sum of distinct units
of $K$. We extend a method due to Thuswaldner and Ziegler \cite{Thuswaldner:2011} that previously did not work for totally complex fields and apply our results
to the case of totally complex quartic number fields.
\end{abstract}

\maketitle

\section{Introduction}

Jacobson~\cite{Jacobson:1964} observed in the 1960's that the two number fields
$\Q(\sqrt{2})$ and $\Q(\sqrt{5})$ share the property that every algebraic integer is the sum of distinct units.
Moreover, he conjectured that these two quadratic number fields are the only quadratic number fields with this
property. Let us call a field with this property a distinct unit generated field or DUG-field for short.

In the 1970's \'Sliwa~\cite{Sliwa:1974} solved this problem for quadratic
number fields and showed that even no pure cubic number field is DUG. These
results have been extended to cubic and quartic fields by
Belcher~\cite{Belcher:1974,Belcher:1976}. In particular, Belcher solved the
case of imaginary cubic number fields completely \cite{Belcher:1976}.

The problem of characterizing all number fields in which every algebraic
integer is a sum of distinct units is still unsolved. Let us note that this
problem is contained in Narkiewicz' list of open problems in his famous
book~\cite[see page 539, Problem~18]{Narkiewicz:1974}.

Recently Thuswaldner and Ziegler \cite{Thuswaldner:2011} used methods originating from the theory of number systems and enumeration and
obtained a new approach to the problem and introduced the following definition in order to measure how far is a number field
away from being a DUG-field.

\begin{definition}
Let $\ord$ be some order in a number field $K$ and $\alpha\in \ord$. Suppose $\alpha$ can be written as a linear combination of units
\[\alpha=a_1\epsilon_1 +\cdots+a_\ell \epsilon_\ell ,\]
such that $\epsilon_1,\ldots,\epsilon_\ell \in \ord^*$ are all distinct and
$a_1\geq \cdots \geq a_\ell>0$ are positive integers. Choose a representation with $a_1$ minimal, then we call
$\omega(\alpha)=a_1$ the unit sum height of $\alpha$. Moreover we define $\omega(0)=0$ and $\omega(\alpha)=\infty$ if $\alpha$ is not the sum of units.

We define
\[\omega(\ord)=\max\{\omega(\alpha) \: :\: \alpha\in\ord\}\]
if the maximum exists. If it does not exist we write
$\omega(\ord)=\omega$ in case of $\ord$ is generated by units and $\omega(\ord)=\infty$ otherwise.

In case of $\ord$ is the maximal order of $K$ we also write $\omega(K)=\omega(\ord)$.
\end{definition}

Unfortunately the method of Thuswaldner and Ziegler \cite{Thuswaldner:2011} only works for number fields which have a real embedding, i.e.\ which are not 
totally complex. Such fields contain a Pisot unit, which is essential for the tool provided there. Recall that an algebraic integer $\alpha>1$ is a Pisot 
number, if all its conjugates are of modulus less than 1. 

On the other hand, Hajdu and Ziegler~\cite{Hajdu:2013} focused on totally complex fields. For the case of quartic totally complex fields they provided the 
following list of candidates of DUG fields, where $\zeta_\mu$ denotes a primitive $\mu$-th root of unity.:

\begin{table}[ht]
\caption{Candidates for totally complex quartic DUG fields. Markers $\dag$ and $\ddag$ are necessary for the statement of 
Theorem~\ref{Th:main}.}\label{Tab:candidates}
%
\begin{itemize}
\item $\Q(\zeta_\mu)$ where $\mu=5,8,12$ or,
\item $\Q(\gamma)$ where $\gamma$ is the root of one of the polynomials $X^4-X+1$, $X^4+X^2-X+1$, $X^4+2X^2-2X+1^\dag$, $X^4-X^3+X+1^\ddag$,
$X^4-X^3+X^2+X+1^\ddag$, $X^4-X^3+2X^2-X+2^\dag$ or,
\item $\Q(\sqrt{a+b\zeta_4})$, with $(a,b)=(1,1),(1,2),(1,4),(7,4)^\dag$ or,
\item $\Q(\sqrt{a+b\zeta_3})$, with $(a,b)=(2,1)$, $(4,1)$, $(8,1)$, $(3,2)$, $(4,3)$, $(7,3)$, $(11,3)$, $(5,4)$, $(9,4)$, $(13,4)$, $(12,5)$, $(11,7)$,
$(9,8)$, $(15,11)$, $(19,11)^\dag$, $(17,12)^\dag$, $(17,16)^\dag$ or,
\item $\Q(\zeta_4,\sqrt{5})$ or $\Q(\zeta_3,\sqrt{d})$, with $d=5,6,21$ or,
\item $\Q\left(\sqrt{-1-\sqrt 2}\right)$ or $\Q\left(\sqrt{-\frac{1+\sqrt 5}2}\right)$.
\end{itemize}
\end{table}

They proved the following theorem.

\begin{theorem}[Hajdu, Ziegler \cite{Hajdu:2013}]\label{Th:complete}
If $K$ is a totally complex quartic field with $\omega(K)=1$, then it is equal to one of the fields in 
the list of Table~\ref{Tab:candidates}. 
\end{theorem}

Note that Hajdu and Ziegler \cite{Hajdu:2013} could not prove that all these fields listed above are DUG-fields. They only succeeded to do so for the fields
$K=\Q(\gamma)$, where
\[\gamma\in\left\{\zeta_5,\zeta_8,\zeta_{12},\sqrt{-1-\sqrt 2},\sqrt{-\frac{1+\sqrt 5}2},\zeta_3+\sqrt 5,\zeta_4+\sqrt 5\right\}\]
or $\gamma$ is a root of the polynomial $X^4+X^2-X+1$ using similar techniques as Belcher \cite{Belcher:1976}. Based on a large computer 
search Hajdu and
Ziegler \cite{Hajdu:2013} conjecture that all the fields in Theorem \ref{Th:complete} are indeed DUG. However, for all the remaining fields the authors
could not even provide a bound for $\omega(K)$.

The aim of this paper is to extend the method of Thuswaldner and Ziegler \cite{Thuswaldner:2011} to totally complex number fields and to apply this method to
extend the list of fields in Theorem \ref{Th:complete} where $\omega(K)=1$ is confirmed. Unfortunately we failed in proving that all fields listed in Theorem
\ref{Th:complete} are distinct unit
generated, but at least we can provide upper bounds for the unit sum height.

\begin{theorem}\label{Th:main}
If $K$ is a totally complex quartic field of the list in Table~\ref{Tab:candidates}
then $\omega(K)\leq 3$. Moreover all such fields are DUG except those marked with $\phantom{}^{\dag}$ or $\phantom{}^{\ddag}$. Those fields marked
with $\phantom{}^\dag$ satisfy at least $\omega(K)\leq 2$ and those marked with $\phantom{}^{\ddag}$ satisfy only $\omega(K)\leq 3$.
\end{theorem}

Connections to positional representation of numbers and recent results on this topic are discussed in the next section. In Section \ref{Sec:UpperBounds} 
we generalize a theorem due to Thuswaldner and Ziegler~\cite[Theorem 2.1]{Thuswaldner:2011} to the case which includes totally complex number fields. The real 
Pisot number is replaced by the notion of complex Pisot number, i.e.\ a non-real algebraic integer $\alpha$ with $|\alpha|>1$ such that the remaining conjugates 
other than $\overline{\alpha}$ lie in the open unit circle. 
With this generalization at hand we consider the case that a totally complex number field $K$ contains a primitive $\mu$-th root of unity
with $\mu>2$. This enables us to prove Theorem \ref{Th:main} up to the second item in the list of fields in Table~\ref{Tab:candidates}. In Section 
\ref{Sec:special} we apply a
variant of our method to the remaining fields and prove Theorem \ref{Th:main} up to the case that $K=\Q(\gamma)$, where $\gamma$ is a root of $X^4-X+1$. This
special case is solved in the last section of the paper by a combinatorial approach.

\section{Connection to positional representation of numbers}\label{Sec:context}

As pointed out already in~\cite{Thuswaldner:2011}, the problem of determining $\omega(K)$ is connected to non-standard positional representation of numbers. 
Consider a field $K$ of unit rank $1$, i.e. $K$ is either a totally real quadratic field or a cubic field with signature $(1,1)$ or a totally complex 
quartic field. By Dirichlet's theorem, all units in $K$ are of the form $\zeta_\mu^{i}\epsilon^{j}$, $i,j\in\Z$, where $\epsilon$ is the fundamental unit and 
$\zeta_\mu^i$ for $1\leq i\leq \mu$ form a finite set of all roots of unity in $K$.

The fact that $\omega(K)\leq w$ can be rephrased by saying that every element of $\ord$ can be represented as
$\sum_{j=l}^{k}a_j\epsilon^j$, where the `digits' $a_j$ take values in the finite set   
\[
\Sigma=\Sigma_\mu(w):=\left\{\sum_{i=1}^{\mu}d_i\zeta_\mu^i \: :\: 0\leq d_i \leq w \;\; \text{for} \;\; 1 \leq i \leq \mu\right\}.
\]

Assume that the fundamental unit $\epsilon$ also generates the integral basis of the ring of integers in $K$, i.e.\ $\ord=\Z[\epsilon]$.
Then the question reformulates to asking whether the set of numbers with finite expansion in base $\epsilon$ with digits in $\Sigma$
satisfies 
\begin{equation}\label{eq_finproperty}
{\rm Fin}_{\Sigma} (\epsilon):= \left\{\sum_{i=l}^ka_i\epsilon^i : k,l\in\Z,\ a_i\in{\Sigma}\right\}= {\mathbb Z}[\epsilon,\epsilon^{-1}]=\Z[\epsilon]\,,
\end{equation}
which will be true, if ${\rm Fin}_{\Sigma} (\epsilon)$ is closed under addition. Indeed, as ${\rm Fin}_{\Sigma} (\epsilon)$ contains $\epsilon^k$ for any $k\in\Z$, one obtains by addition the whole ring $\Z[\epsilon]$.
This is a generalisation of the so-called finiteness property studied in numeration systems, first introduced for R\'enyi $\beta$-expansions of real numbers by 
Frougny and Solomyak~\cite{Frougny-Solomyak:1992}. 

Another similar problem is the height reducing property (HRP) of numbers $\alpha$, where however, one requires that elements 
of $\Z[\alpha]$ rewrite with digits in a finite set which are non-zero only at non-negative powers of $\alpha$. Characterization of numbers satisfying such 
property was recently completed by Akiyama, Thuswaldner and Za\"\i mi~\cite{Akiyama-Zaimi:2013} and in~\cite{Akiyama-Thuswaldner-Zaimi:2014}, where the authors 
show that a complex $\alpha$ has HRP if and only if it is an algebraic integer whose conjugates over $\Q$ are either all of modulus one, or all of modulus 
greater than one. 

\section{Computing upper bounds for the unit sum height}\label{Sec:UpperBounds}

Before we state and prove our main tool (see Theorem \ref{Th:USHUppBound} below) to compute upper bounds for the unit sum height we have to introduce some
notation.

Let $K$ be a number field of degree $2s+t$, signature $(t,s)$ and let $\ord$ be an order of $K$. 
Also let us fix the real embeddings $\sigma_{1}$, \dots, $\sigma_{t}$ and the complex embeddings
$\sigma_{t+1}=\bar\sigma_{t+s+1},\ldots,\sigma_{t+s}=\bar\sigma_{t+2s}$ of $K$. For $\alpha\in K$ we denote by
$\alpha^{(i)}=\sigma_i\alpha$ the Galois conjugates of $\alpha$ and let us identify $\sigma_{t+1}(K)$ with $K$ and $\sigma_{t+1}(\ord)$ with $\ord$ 
respectively.

Let $\epsilon\in\ord$ be a complex Pisot number, i.e.\ such that $|\epsilon|>1$ and $|\epsilon^{(i)}|<1$ for all $i=1,\ldots,t+s$, $i\neq t+1$. Given a finite 
set $\Sigma\subset\ord$, denote
$$
C_i:=\max\{|c^{(i)}| \,:\, c\in\Sigma\}\,,\qquad  \text{ for } i=1,\dots,t+s,\ i\neq t+1.
$$
Consider a compact set $P\subset\C$, containing at least a neighborhood of $0$ and denote by $\mathcal{B}(\epsilon,\Sigma,P)$ the cylinder defined by
$$
\mathcal{B}(\epsilon,\Sigma,P):=\Big\{\alpha\in\ord \,:\, \alpha\in P\text{ and } |\alpha^{(i)}|\leq \frac{C_i}{1-|\epsilon^{(i)}|}
\text{ for }i=1,\dots,t+s,\, i\neq t+1 \Big\}\,.
$$
Note that since the lattice
\[\Lambda_\ord:=\{(\alpha^{(i)})_{1\leq i\leq t+s} \: :\: \alpha \in \ord\}\subset \R^t\times\C^s\]
is discrete, the set ${\mathcal B}(\epsilon,\Sigma,P)$ is finite.

Now we have all the notations to state the main result of this section:

\begin{theorem}\label{Th:USHUppBound}
Let $\epsilon\in\ord$ be a complex Pisot number. With the notation above, assume that
 \begin{equation}\label{Eq:Covering}
  \epsilon P \subset \bigcup_{s\in\Sigma}(s+P).
 \end{equation}
Then for each $\alpha\in\ord$ there exist $N,n\in \mathbb{N}$ such that
\begin{equation}\label{Eq:Rep1}
\alpha \epsilon^N=\beta+\sum_{i=0}^n c_{i} \epsilon^i,
\end{equation}
with $c_i\in\Sigma$ and $\beta$ is contained in the finite set $\mathcal{B}(\epsilon,\Sigma,P)$. The elements of $\mathcal{B}(\epsilon,\Sigma,P)\setminus\{0\}$
will be called critical points.
\end{theorem}

\begin{proof}
 Let $x \in \C$ and assume that \eqref{Eq:Covering} holds and let $n\geq -1$ be an integer such that $x\in \epsilon^{n+1}P$. Such an integer $n$ exists since 
$P$ contains a neighborhood of $0$. Let us prove by induction on $n$ that there exist $c_0,\dots,c_{n}\in\Sigma$ such that
\begin{equation}\label{eq:Reduce} x-\sum_{j=0}^{n}c_j \epsilon^j \in P.\end{equation}

The case $n=-1$ is trivial.
Now let us assume that \eqref{eq:Reduce} is proved for all integers $M \leq n$ and assume that $x\in \epsilon^{n+1}P$. Since by assumption
\[
x\in \epsilon^{n+1} P \subset \bigcup_{s\in\Sigma}\left(s\epsilon^{n}+\epsilon^{n}P\right)\,,
\]
there exists some $s=c_{n}\in\Sigma$ with $x-s\epsilon^{n} \in \epsilon^{n}P$. Since we assume by induction that \eqref{eq:Reduce} is true for all $M\leq n$ we 
know that there exist
$c_0,\dots,c_{n-1}\in\Sigma$ such that
\[
x-s\epsilon^{n}-\sum_{j=0}^{n-1} c_j \epsilon^j=x- \sum_{j=0}^{n} c_j \epsilon^j\in P
\]
and we have established \eqref{eq:Reduce}.

With \eqref{eq:Reduce} at hand, we are ready to prove Theorem \ref{Th:USHUppBound}.
Let $\alpha \in \ord$ be arbitrary. Since $|\epsilon|>1$ and for all other conjugates we have $|\epsilon^{(i)}|<1$, there exists for each $\delta>0$ a
non-negative integer $N$ such that $\left|\alpha^{(i)}(\epsilon^{(i)})^N\right|<\delta$ for $i=1,\ldots,t+s$, $i\neq t+1$. In other words, apart form
$\alpha\epsilon^N$ and $\overline{\alpha}\overline{\epsilon}^N$, all conjugates of $\alpha\epsilon^N$ are small. In view of \eqref{eq:Reduce} we can 
approximate $\alpha\epsilon^N$ by a
combination of powers of $\epsilon$ with coefficients in $\Sigma$. In particular, we apply \eqref{eq:Reduce} to $x=\alpha\epsilon^N$. This yields $n\in\N$ 
and
$c_j\in\Sigma$ for $0\leq j \leq n$
such that
\begin{equation}\label{firstbound}
\beta:=\alpha\epsilon^N-\sum_{j=0}^{n}c_j \epsilon^j \in P.
\end{equation}
Then, taking conjugates, we get
\begin{equation}\label{secondbound}
|\beta^{(i)}|= \left|\alpha^{(i)}(\epsilon^{(i)})^N-\sum_{j=0}^{n}c_j^{(i)} (\epsilon^{(i)})^j\right| \leq  \delta + C_i \sum_{j=0}^{n} |(\epsilon^{(i)})^j|
<\delta+\frac{C_i}{1-|\epsilon^{(i)}|}
\end{equation}
for $1\leq i \leq t+s$, $i\neq t+1$. Since $\Lambda_\ord$ is a discrete set and since we assume that $P$ is compact there exists a $\delta_0$ such that for 
every
$0<\delta<\delta_0$ the conditions \eqref{firstbound}
and \eqref{secondbound} imply $\beta \in {\mathcal B}(\epsilon,\Sigma,P)$.
\end{proof}

\section{Application to fields with a fourth or sixth root of unity}\label{Sec:Appl}

Assume that $K$ contains a $\mu$-th root of unity with $\mu>2$ and denote by $\zeta_\mu$ some
primitive $\mu$-th root of unity. We may assume that $\mu$ is even. Indeed if $\mu$ is odd then with $\zeta_\mu$ also $-\zeta_\mu=\zeta_{2\mu}$ is an
element of $K$. As explained in Section~\ref{Sec:context}, the role of the digit set $\Sigma$ will be taken by the set of all possible sums of roots of unity 
with bounded coefficients. Therefore we
write
\[
\Sigma=\Sigma_\mu(w):=\left\{\sum_{i=1}^{\mu}d_i\zeta_\mu^i \: :\: 0\leq d_i \leq w \;\; \text{for} \;\; 1 \leq i \leq \mu\right\}.
\]

First, let us assume that $K$ is a complex (not necessarily quartic) field that contains a fourth root of unity. Given a complex Pisot number $\epsilon\in 
\ord$, we apply Theorem~\ref{Th:USHUppBound}
to the case where $P\subset \C$ is the square with vertices $\frac{\pm 1\pm \imath}2$ and obtain a simple criterion such that the covering
property~\eqref{Eq:Covering} holds:

\begin{lemma}\label{lem:criterion_mu_4}
 Let $P\subset \C$ be the square with vertices $\frac{\pm 1\pm \imath}{2}$. Let $\eta=\epsilon\frac{1+\imath}2$, then~\eqref{Eq:Covering} is satisfied, provided
 \[\max\{|\Re(\eta)|,|\Im(\eta)|\}\leq \frac{1+2w}2.\]
\end{lemma}

\begin{proof}
 Note that
 \[\bigcup_{s\in\Sigma_\mu(w)}s+P=(1+2w) P.\]
 Since $P$ is convex, it suffices to prove that all the vertices of $\epsilon P$ lie within the square $(1+2w)P$, i.e.
 \[\max\{|\Re(\eta)|,|\Im(\eta)|\}\leq \frac{1+2w}2.\]
\end{proof}

In view of Theorem \ref{Th:complete} we want to apply Lemma \ref{lem:criterion_mu_4} together with Theorem \ref{Th:USHUppBound} to the fields
$\Q(\sqrt{1+\zeta_4})$, $\Q(\sqrt{1+2\zeta_4})$, $\Q(\sqrt{1+4\zeta_4})$ and $\Q(\sqrt{7+4\zeta_4})$. For the complex Pisot number $\epsilon$ we take the 
fundamental unit of $K$. Since the computations in all cases are similar, we
only give details
for the case $K=\Q(\sqrt{1+\zeta_4})$. For some details in the other cases see Table \ref{Tab:fourth} below.

Let us discuss the case that $K=\Q(\sqrt{1+\zeta_4})$ and $\ord$ is the maximal order of $K$. We write $\gamma=\sqrt{1+\zeta_4}$ and choose a branch of the
logarithm
such that $\Re(\gamma)>0$. The fundamental unit is $\epsilon=1+\gamma$ and
\[\epsilon\frac{1+\imath}2=\eta\simeq 0.822+ 1.277 \imath.\]
Due to Lemma \ref{lem:criterion_mu_4} we may apply Theorem \ref{Th:USHUppBound} with $w=1$ and we
obtain the critical points
\[\mathcal B(\epsilon,\Sigma,P)=\left\{\alpha \in \ord\: :\: |\Re(\alpha)|,|\Im(\alpha)|\leq \frac 12,\; \left|\alpha^{(2)}\right|\leq 2.647\right\}\]
which are exactly $\zeta_4^k(1-\gamma)$, with $k=0,1,2,3$. But, these critical points are $\zeta_4^k \epsilon^{-1}$ with $k=0,1,2,3$ written in terms of
$\epsilon^{-1}$.
Since all the critical points can be written as the sum of distinct units such that the exponent of $\epsilon$ is negative we deduce from
Theorem~\ref{Th:USHUppBound} that each algebraic integer of $K$ is the sum of distinct units, i.e. $K$ is DUG.

In the other cases the critical points are less obvious. But by a computer search we were able to confirm that all critical points can be written in the form
$\sum_{k=-1}^{-B}s_i\epsilon^i$ with $s_i\in\Sigma_4(w)$. Further, let us denote by $C$ the number of critical points.

\begin{table}[ht]
\caption{Details to the computations in case that $K$ contains fourth roots of unity.}\label{Tab:fourth}
\begin{tabular}{|c|c|c|c||c|c|c|c|}
\hline $K$ & $w$ & $C$ & $B$ & $K$ & $w$ & $C$ & $B$\\\hline\hline
$\Q(\sqrt{1+\zeta_4})$ & $1$ & $4$ & $1$ & $\Q(\sqrt{1+4\zeta_4})$ & $1$ & $16$ & $2$ \\\hline
$\Q(\sqrt{1+2\zeta_4})$ & $1$ & $8$ & $2$ & $\Q(\sqrt{7+4\zeta_4})$ & $2$ & $8$ & $2$\\\hline
\end{tabular}
\end{table}

\begin{remark}
It is rather plausible that another choice of $P$ might yield $w=1$ in the case that $K=\Q(\sqrt{7+4\zeta_4})$. The best choice for $P$ seems to be the
unique compact set $P$ which satisfies
\[P= \bigcup_{s\in\Sigma}f_s(P),\]
where $f_s(x)=s+\frac{x}{\epsilon}$ for all $s\in\Sigma_4(1)$ (e.g. see \cite[Theorem 9.1]{Falconer:FG}). Obviously this set $P$ is compact and satisfies
condition \eqref{Eq:Covering} of Theorem \ref{Th:USHUppBound}. Unfortunately this iterated function system does not fulfill the so called ``open set condition''
(see e.g. \cite[page 118]{Falconer:FG}) and therefore we are unable to show that $P$ contains a neighborhood of $0$, which is essential in the proof of
Theorem~\ref{Th:USHUppBound}.
\end{remark}

Now let us assume that $K$ is a complex (not necessarily quartic) field that contains sixth roots of unity. In this case we choose $P$ to be a regular
hexagon. Again, $\epsilon$ is the fundamental unit.

\begin{lemma}\label{lem:criterion_mu_6}
 Let $P\subset \C$ be the hexagon with vertices $v_k=\frac1{\sqrt3}{\exp\left(\frac{2\pi \imath (2k+1)}{12}\right)}$ and $k=0,\ldots,5$. Let $\eta_k=\epsilon 
v_k$
then \eqref{Eq:Covering} is satisfied, provided
 \begin{equation}\label{Cond:mu_6}
 \max_{0\leq k \leq 5}\{|\Im(\eta_k)|\}\leq \frac{5w+2}{2\sqrt{3}}.
 \end{equation}
\end{lemma}

\begin{proof}
 Note that
 \[\bigcup_{s\in\Sigma_\mu(w)}s+P\supset \exp (\imath \pi/6)\frac{5w+2}{\sqrt{3}} P.\]
 Therefore it suffices to prove that all the vertices of $\epsilon P$ lie within the hexagon $\frac{5w+2}{2}\exp(\pi \imath/6) P$, i.e.
 \[\max_k\{|\Im(\eta_k)|\}\leq \frac{5w+2}{2\sqrt{3}}.\]
 \begin{figure}
  \centering
  \includegraphics[width=\linewidth]{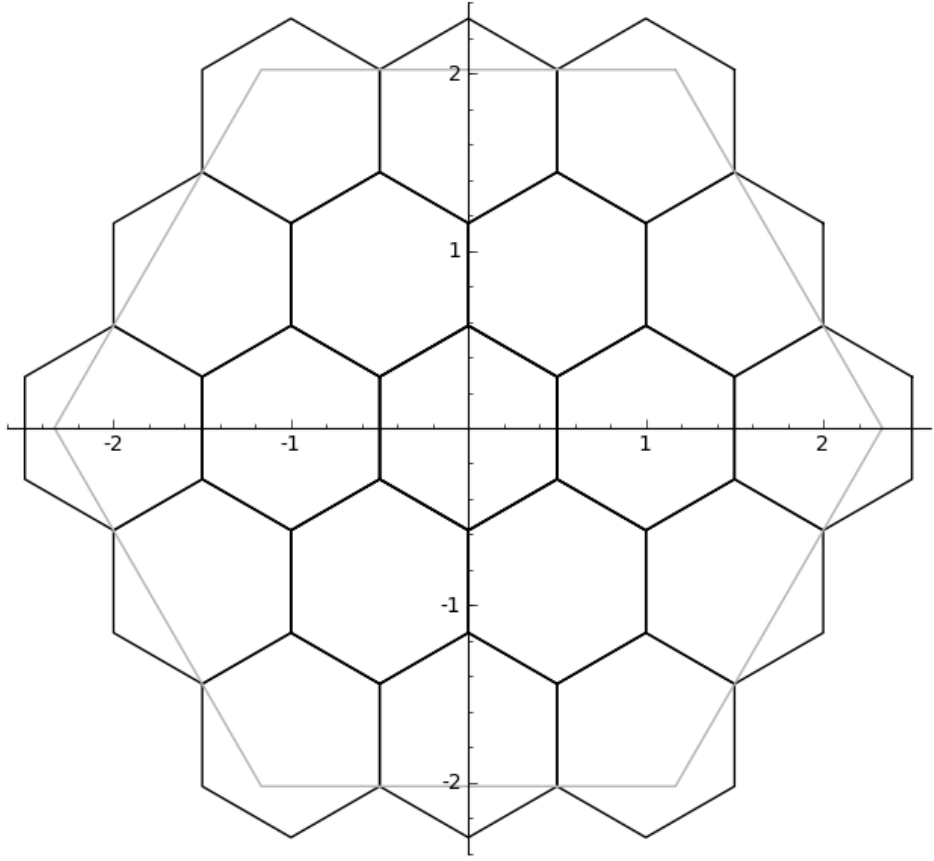}
  \caption{The case that $w=1$ in the proof of Lemma \ref{lem:criterion_mu_6}}
  \label{fig:hexagons}
\end{figure}

For a better illustration see Figure \ref{fig:hexagons}, where the case that $w=1$ is shown. The black hexagons are translations of $P$ by all possible
$s\in\Sigma_6(1)$. The gray hexagon is the hexagon $\exp (\imath \pi/6)\frac{5w+2}{\sqrt{3}}P=\exp (\imath \pi/6)\frac{7}{\sqrt{3}}P$.
\end{proof}

We proceed as described in the case that $K$ contains a fourth root of unity. Since the computations are similar to those made in the case that $\mu=4$ we
only give a few details (see Table \ref{Tab:sixth} below).

\begin{table}[ht]
\caption{Details to the computations in case that $K$ contains sixth roots of unity}\label{Tab:sixth}
\begin{tabular}{|c|c|c|c||c|c|c|c|}
\hline $K$ & $w$ & $C$ & $B$ & $K$ & $w$ & $C$ & $B$\\\hline\hline
$\Q(\sqrt{2+\zeta_3})$ & $1$ & $6$ & $1$ & $\Q(\sqrt{12+5\zeta_3})$ & $1$ & $6$ & $1$  \\\hline
$\Q(\sqrt{4+\zeta_3})$ & $1$ & $66$ & $3$ & $\Q(\sqrt{11+7\zeta_3})$ & $1$ & $0$ &  \\\hline
$\Q(\sqrt{8+\zeta_3})$ & $1$ & $6$ & $1$ & $\Q(\sqrt{9+8\zeta_3})$ & $1$ & $6$ & $1$ \\\hline
$\Q(\sqrt{3+2\zeta_3})$ & $1$ & $0$ &  & $\Q(\sqrt{15+11\zeta_3})$ & $1$ & $0$ &  \\\hline
$\Q(\sqrt{4+3\zeta_3})$ & $1$ & $0$ &  & $\Q(\sqrt{19+11\zeta_3})$ & $2$ & $6$ & $1$\\\hline
$\Q(\sqrt{7+3\zeta_3})$ & $1$ & $6$ & $1$ & $\Q(\sqrt{17+12\zeta_3})$ & $2$ & $6$ & $1$\\\hline
$\Q(\sqrt{11+3\zeta_3})$ & $1$ & $0$ &  & $\Q(\sqrt{17+16\zeta_3})$ & $2$ & $6$ & $1$ \\\hline
$\Q(\sqrt{5+4\zeta_3})$ & $1$ & $24$ & $2$ & $\Q(\zeta_3,\sqrt{6})$ & $1$ & $0$ &  \\\hline
$\Q(\sqrt{9+4\zeta_3})$ & $1$ & $6$ & $1$ & $\Q(\zeta_3,\sqrt{21})$ & $1$ & $6$ & $1$ \\\hline
$\Q(\sqrt{13+4\zeta_3})$ & $1$ & $0$ & & & & & \\\hline
\end{tabular}
\end{table}

\begin{remark}
Figure \ref{fig:hexagons} shows that Lemma \ref{lem:criterion_mu_6} is not optimal and it may happen that $\epsilon P\subset \bigcup_{s\in\Sigma_6(w)}s+P$
holds but condition \eqref{Cond:mu_6} in Lemma \ref{lem:criterion_mu_6} fails. Although this seems to be rather unlikely, we checked in case that $K$ is one of
the fields $\Q(\sqrt{19+11\zeta_3})$, $\Q(\sqrt{17+12\zeta_3})$ or $\Q(\sqrt{17+16\zeta_3})$ whether $\epsilon P\subset \bigcup_{s\in\Sigma_6(1)}s+P$ holds
although Lemma \ref{lem:criterion_mu_6} fails with $w=1$. But, in all three cases condition \eqref{Eq:Covering} fails for $w=1$.
\end{remark}

\begin{remark}
 In case that $\mu>6$ it is not hard to get criteria which are similar to the criteria in Lemmas \ref{lem:criterion_mu_4} and \ref{lem:criterion_mu_6} such
that the covering property
\eqref{Eq:Covering} holds. As we can already see in the case that $\mu=6$ such results are either not best possible or not very simple. So in view of Theorem
\ref{Th:main} we abandon to discuss criteria for $\mu>6$.
\end{remark}

\begin{remark}
We want to note that in case of $\mu=8$ and $K=\Q(\zeta_8)$ Theorem \ref{Th:main} yields a new proof that $\Q(\zeta_8)$ is indeed DUG. Indeed
choose $P$ to be the square with vertices $\frac{\pm 1\pm \imath}2$. Then it is easy to show that
\[\epsilon P=(1+\sqrt{2})P\subset \bigcup_{s\in\Sigma_8(1)}(s+P).\]
Since the critical points in this case are $\frac{\zeta_8^{2k+1}}{\epsilon}$ for $k=0,1,2,3$ we see that $K$ is indeed DUG.
\end{remark}

\section{Five special cases}\label{Sec:special}

Now we consider the remaining five number fields in Theorem~\ref{Th:main}, namely those which do not contain any roots of unity $\zeta_\mu$ for $\mu>2$.
The same approach as in
the previous section will not lead to success, since the alphabet $\Sigma_\mu(w)=\{-w,\dots,0,\dots,w\}$ is contained in the real line.
Instead, we take the digit set
$\Sigma=\left\{d_0+d_1\tilde{\epsilon} \: :\: -w\leq d_0,d_1 \leq w \right\}$, and expand the number $\alpha\in\ord$ in base $\epsilon = \tilde{\epsilon}^2$,
where $\tilde{\epsilon}$ is a fundamental unit with $|\tilde{\epsilon}|>1$. The compact set $P\subset \C$ is taken to be the parallelogram with vertices
$\frac{\pm1\pm{\tilde{\epsilon}}}{2}$.

\begin{lemma}\label{lem:criterion_mu_2}
 Let $P\subset \C$ be the parallelogram with vertices $\frac{\pm1\pm{\tilde{\epsilon}}}{2}$. Let
 $$
 A=\left(\begin{smallmatrix}1&\Re(\tilde{\epsilon})
 \\0&\Im(\tilde{\epsilon})\end{smallmatrix}\right),
 \qquad
 \begin{array}{ll}
 a_1=\Re\big(\frac{\tilde{\epsilon}^2}2(1+\tilde{\epsilon})\big), &b_1=\Im\big(\frac{\tilde{\epsilon}^2}2(1+\tilde{\epsilon})\big),\\[2mm]
 a_2=\Re\big(\frac{\tilde{\epsilon}^2}2(1-\tilde{\epsilon})\big), &b_2=\Im\big(\frac{\tilde{\epsilon}^2}2(1-\tilde{\epsilon})\big).
 \end{array}
 $$
 Then \eqref{Eq:Covering} is satisfied, provided
 $$
 \max_{k=1,2}\left\{\left|(1,0)A^{-1}\left(\begin{smallmatrix}a_k\\b_k\end{smallmatrix}\right)\right|,
 \left|(0,1)A^{-1}\left(\begin{smallmatrix}a_k\\b_k\end{smallmatrix}\right)\right|\right\}\leq \frac{1+2w}2\,.
 $$
\end{lemma}

\begin{proof}
 Note that
 \[\bigcup_{s\in\Sigma}(s+P)=(1+2w) P.\]
 Therefore it suffices to prove that all the vertices of $\epsilon P$, namely $\frac{\tilde{\epsilon}^2}2(\pm1\pm\tilde{\epsilon})$,
 lie within the parallelogram $(1+2w) P$. In order to check this, it is convenient to consider $1,\tilde{\epsilon}$ as a basis of $\C$ over $\R$ instead of
$1,\imath$. If $A$ is as above, we have $z=a+b\imath=c+d\tilde{\varepsilon}$, where
$ \left(\begin{smallmatrix}c
 \\d\end{smallmatrix}\right) = A^{-1} \left(\begin{smallmatrix}a
 \\b\end{smallmatrix}\right)$.
 Hence $z$ lies within $(1+2w) P$, if $|c|,|d|\leq \frac{1+2w}2$.
\end{proof}

Let us note that the bound for $w$ obtained in Lemma \ref{lem:criterion_mu_2} depends on which embedding $K\hookrightarrow \C$ we chose.
For instance in the case that $K$ is the number field with minimal polynomial $X^4+2X^2-2X+1$ one obtains either $w=2$ or $w=4$ depending on the choice of the
actual embedding $K\hookrightarrow\C$. However since the unit sum height $\omega(K)$ does not depend on the embedding we can choose $\tilde \epsilon$ such that
in view of Lemma~\ref{lem:criterion_mu_2} the quantity $w$ is minimal.

Once we have chosen the optimal embedding $K\hookrightarrow \C$ we can proceed as before and we only give a few details on the applications of
Lemma~\ref{lem:criterion_mu_2} and Theorem~\ref{Th:USHUppBound}. In particular, see Table \ref{Tab:special} below for details.

\begin{table}[ht]
\caption{Details to the computations in case $K$ has the following minimal polynomial.}\label{Tab:special}
\begin{tabular}{|c|c|c|c|c|}
\hline minimal polynomial & $w$ & $\tilde\epsilon$ & $C$ & $B$ \\\hline\hline
$X^4-X+1$ & $2$ & $0.727+0.934\imath$ & $112$ & $7$ \\\hline
$X^4-X^3+X^2+X+1$ & $3$ & $-0.933 + 1.132\imath$ &  $42$ & $4$ \\\hline
$X^4-X^3+X+1$ & $3$ & $-1.066 + 0.864\imath$ & $56$ & $4$ \\\hline
$X^4+2X^2-2X+1$ & $2$ & $0.475+1.509\imath$ & $18$ & $4$ \\\hline
$X^4-X^3+2X^2-X+2$ & $2$ & $0.204+1.664\imath$ & $12$ & $3$ \\\hline
\end{tabular}
\end{table}

\section{A combinatorial approach}\label{Subsec:combinatorial}

The aim of this section is to prove that $K=\Q(\gamma)$ is DUG, where $\gamma$ is a root of the polynomial $X^4-X+1$. Although we already
proved in the previous section that $\omega(K)\leq 2$ we do not assume this result in this section. Independently from the rest of the paper we prove:

\begin{proposition}\label{thm_combinatorial_dug}
The field $K=\Q(\gamma)$ with $\gamma$ being a root of the polynomial $X^4-X+1$ is DUG.
\end{proposition}

Since the maximal order $\ord$ of $K$ is of the form $\ord=\Z[\gamma]=\Z[\gamma,\gamma^{-1}]$ we can write every element $\alpha\in\ord$ in the form
\begin{equation}\label{Eq:Representationalpha}
 \alpha=\sum_{n=-\infty}^{\infty} v_n\gamma^n
\end{equation}
with $v_n\in\Z$ and $v_n\neq 0$ for at most finitely many indices. Such a $\gamma$-representation of $\alpha$ is sometimes written
$$
\alpha=\cdots v_2v_1v_0{\scriptstyle \bullet} v_{-1} v_{-2}\cdots\,,
$$
where the fractional point ${\scriptstyle \bullet}$ separates between the coefficients at negative and non-negative powers of the base $\gamma$. We are only
interested in the fact that non-vanishing coefficients in the $\gamma$-representation are finitely many. Thus we will abbreviate representation
\eqref{Eq:Representationalpha} by the finite word
$v_kv_{k-1}\cdots v_{\ell+1}v_\ell$, where the indices $k$ and $\ell$ are such that $v_n=0$ for all $n>k$ and all $n<\ell$, without marking the fractional
point. Note that the $\gamma$-representation is not unique. Since $\gamma^n(\gamma^4-\gamma+1)=0$ for all $n$, position-wise  addition or
subtraction of $100\bar 1 1$, with $\bar 1=-1$, at any position does not change the value of $\alpha$ but only its $\gamma$-representation, i.e. the words
$v_k\cdots v_nv_{n-1}v_{n-2}v_{n-3}v_{n-4}\cdots v_\ell$ and $v_k\cdots (v_n+1)v_{n-1}v_{n-2}(v_{n-3}-1)(v_{n-4}+1)\cdots v_\ell$ represent the same element
$\alpha\in\ord$.

From this point of view any element $\alpha\in\ord$ has some $\gamma$-representation
\begin{equation}\label{Eq:representation_overZ}
x_3x_2x_1x_0\,,\quad x_i\in\Z\,
\end{equation}
and if $\alpha$ is also a sum of distinct units, there exists another $\gamma$-representation of the form
\begin{equation}\label{Eq:representation_smalldigits}
v_kv_{k-1}\cdots v_0v_{-1}\cdots v_\ell \,, \quad v_i\in\{\overline{1},0,1\}, \ell,k\in\Z\,.
\end{equation}
Hence, if we want to prove that the field $K$ is DUG, we have to show that any representation of the form \eqref{Eq:representation_overZ} can be rewritten
into \eqref{Eq:representation_smalldigits} without changing the value of the represented number.

\begin{definition}
Let $\mathcal{A}\subseteq\Z$ be an alphabet and let $w\in\mathcal{A}^*$ be a finite word. We say that the word $u\in\mathcal{A}^*$ can be rewritten by
$w$ to $v\in\mathcal{A}^*$, if it is possible to obtain $v$ from $u$ by finitely many position-wise additions or subtractions of shifts of $w$. We
denote this by $u\leftrightarrow_w v$ or just $u\leftrightarrow v$, if $w$ is understood. In this context we call $w$ the rewriting rule.

Moreover let $v=v_k\cdots v_\ell \in\mathcal A^*$ be a finite word, then we denote by
\[W(v)=\sum_{n=\ell}^k |v_n|\]
the weight of $v$.
\end{definition}

Let us note that the symbol $\Z^*$ bears some ambiguity. It may be the set of finite words with alphabet $\Z$ or it may denote the set $\{\pm 1\}$, which is
the
group of units of $\Z$. Since from the context the meaning of $\Z^*$ is always clear in this paper we allow this ambiguity.

In view of Proposition \ref{thm_combinatorial_dug} let us fix $w=100\bar 1 1$. If $u$ and $v$ are $\gamma$-representations with $u\leftrightarrow_w v$,
then $u$ and $v$ represent the same element $\alpha\in\ord$, as explained above. Hence Proposition \ref{thm_combinatorial_dug} is equivalent to the following:

\begin{proposition}\label{thm_combinatorial_dugV2}
For every word $x_3x_2x_1x_0\in\Z^*$ there exists a word $v\in\{\overline{1},0,1\}^*$ such that
\[x_3x_2x_1x_0 \leftrightarrow_w v \,.\]
\end{proposition}

Note that for two digits $a$ and $b$, we denote by $ab$ their concatenation and  by $a\cdot b$ standard multiplication. In order to prove Proposition
\ref{thm_combinatorial_dugV2} we need the following lemma.

\begin{lemma}\label{lem_getsparse}
Let $u\in\Z^*$ be a finite word. Then $u\leftrightarrow_{w} v=v_kv_{k-1}\cdots v_{\ell+1}v_\ell$, where $v$ fulfills the following conditions:
\begin{enumerate}
\item[(i)] $v_i\in\{\overline{2},\overline{1},0,1,2\}$ for all $i\in\{k,\dots,\ell\}$
\item[(ii)] $v_{i+m}\cdot v_i>0 \Rightarrow m\notin\{1,2,4\}$
\item[(iii)] $v_{i+m}\cdot v_i<0 \Rightarrow m\notin\{1,3\}$
\item[(iv)] $v_{i+2}\cdot v_i<0 \Rightarrow v_{i+4}=v_{i+5}=0$
\item[(v)] $v_{i+3}\cdot v_i>0 \Rightarrow v_{i+6}=0$
\end{enumerate}
\end{lemma}

\begin{proof}
Multiple application of the original rewriting rule $w=w_1=100\overline{1}1$ gives rise to other useful ones, in particular $w_2=10000003000001$,
$w_3=100010010\overline{1}1$ and $w_4=11100001$. We prove the lemma by showing that these rewriting rules can be used in such a manner, that they
decrease the weight $W(u)$ of the word $u$ in every rewriting step until $u$ satisfies the conditions of the lemma. We apply the rewriting rules $w_1,w_2,w_3$
or $w_4$ in the following situations (the underlined digits indicate which digits we want to ``reduce'' in order to obtain a smaller weight):
\begin{enumerate}
\itemsep 2mm
\item[(a)] If $|v_i|\geq 3$, then
\[v_{i+7}v_{i+6}\cdots \underline{v_i}\cdots v_{i-5}v_{i-6} \leftrightarrow_{w_2} (v_{i+7}\pm 1)v_{i+6}\cdots \underline{(v_i\pm 3)}\cdots v_{i-5}(v_{i-6}\pm
1).\]

\item[(b)]
If $v_{i+1}\cdot v_i<0$, then
\[v_{i+4}v_{i+3}v_{i+2}\underline{v_{i+1}v_i} \leftrightarrow_{w_1} (v_{i+4}\mp 1)v_{i+3}v_{i+2}\underline{(v_{i+1}\pm 1)(v_i \mp 1)}.\]

\item[(c)]
If $v_{i+3}\cdot v_i<0$, then
\[\underline{v_{i+3}}v_{i+2}v_{i+1}\underline{v_i}v_{i-1} \leftrightarrow_{w_1} \underline{(v_{i+3}\pm 1)}v_{i+2}v_{i+1}\underline{(v_i\mp 1)}(v_{i-1} \pm
1).\]

\item[(d)]
If $v_{i+4}\cdot v_i>0$, then
\[\underline{v_{i+4}}v_{i+3}v_{i+2}v_{i+1}\underline{v_i} \leftrightarrow_{w_1} \underline{(v_{i+4}\pm 1)}v_{i+3}v_{i+2}(v_{i+1}\mp
1)\underline{(v_i \pm 1)}.\]

\item[(e)]
If $v_{i+2}\cdot v_i<0$ and $v_{i+5}\cdot v_i<0$, then
\begin{multline*}
v_{i+9}\cdots \underline{v_{i+5}}v_{i+4}v_{i+3}\underline{v_{i+2}}v_{i+1}\underline{v_i}v_{i-1} \leftrightarrow_{w_3}\\
(v_{i+9}\pm 1)\cdots \underline{(v_{i+5}\pm 1)}v_{i+4}v_{i+3}\underline{(v_{i+2}\pm 1)}v_{i+1}\underline{(v_i\mp 1)}(v_{i-1}\pm 1).
\end{multline*}

\item[(f)]
If $v_{i+3}\cdot v_i>0$ and $v_{i+6}\cdot v_i>0$, then
\begin{multline*}
v_{i+10}\cdots \underline{v_{i+6}}v_{i+5}v_{i+4}\underline{v_{i+3}}v_{i+2}v_{i+1}\underline{v_i} \leftrightarrow_{w_3}\\
(v_{i+10}\pm 1)\cdots \underline{(v_{i+6}\pm 1)}v_{i+5}v_{i+4}\underline{(v_{i+3}\pm 1)}v_{i+2}(v_{i+1}\mp 1)\underline{(v_i\pm 1)}.
\end{multline*}

\item[(g)]
If $v_{i+1}\cdot v_i>0$, then
\[\underline{v_{i+1}v_i}v_{i-1}v_{i-2}\cdots v_{i-6} \leftrightarrow_{w_4} \underline{(v_{i+1}\pm 1)(v_i\pm 1)}(v_{i-1}\pm 1)v_{i-2}\cdots (v_{i-6}\pm
1).\]

\item[(h)] If $v_{i+2}\cdot v_i>0$, then
\[\underline{v_{i+2}}v_{i+1}\underline{v_i}v_{i-1}\cdots v_{i-5} \leftrightarrow_{w_4} \underline{(v_{i+2}\pm 1)}(v_{i+1}\pm 1)\underline{(v_i\pm
1)}v_{i-1}\cdots (v_{i-5}\pm 1).\]
\end{enumerate}

Observe that in the first six cases the weight strictly decreases. Let us emphasize here that if an application of (g) does not decrease the weight of the
word, then $v_{i}\cdot v_{i-1}<0$. Therefore we can apply (b) instead of (g) with the index $i$ replaced by $i-1$ and the application of (b) reduces the
weight. Similarly if an application of (h) does not decrease the weight of the word, then $v_{i+1}\cdot v_{i}<0$ and again we can apply (b)
instead of (h). Therefore if an application of the rules (a)--(h) is possible, we can choose an application that strictly decreases the weight. So after
finitely many steps we obtain a word over the alphabet $\{\overline{2},\overline{1},0,1,2\}$ which cannot be further rewritten by the rules (a)--(h). But a word
that cannot be rewritten by any of the rules (a)--(h), satisfies the conditions of the lemma.
\end{proof}

Now let us turn to the proof of Proposition \ref{thm_combinatorial_dugV2}.

\begin{proof}[Proof of Proposition \ref{thm_combinatorial_dugV2}]
According to Lemma \ref{lem_getsparse} we may assume that any $\alpha\in\Z[\gamma]$ is represented by a word $v=v_kv_{k-1}\cdots
v_{\ell+1}v_\ell\in \{\overline{2},\overline{1},0,1,2\}^*$ satisfying the
requirements of Lemma~\ref{lem_getsparse}. We have to show that $v\leftrightarrow u$ with $u\in\{\overline{1},0,1\}^*$. This
will be achieved by reading the word $v$ from left to right and rewriting some of its parts whenever the digit $\pm 2$ is encountered.

The conditions on $v$, i.e. the conditions (i)--(v) of Lemma \ref{lem_getsparse}, imply that for any two consecutive occurrences of digits $\pm 2$, either the
shortest factor of $v$ containing these two digits must
belong to the set
$$
F=\{\pm(20\overline{2}), \pm(2002), \pm(2000\overline{2}), \pm(200002), \pm(20000\overline{2}), \pm(20\overline{1}00\overline{2})\}
$$
or these two occurrences of the digit $\pm 2$ are at least six positions apart. Moreover, also the occurrences of digits $\pm 1$ are severely limited. As we
will consider only a neighborhood $v_{j+7}\cdots v_{j-1}$ for each occurrence $v_j=\pm 2$ and since every $\pm 2$ is contained in a factor $\pm(020)$, it
suffices to rewrite factors from $F$ and ``isolated'' $\pm 2$'s. In order to prove Proposition \ref{thm_combinatorial_dugV2} it is enough to prove the
following claim:

\begin{claim}
 Let $v=v_k\dots v_\ell$ be a word satisfying the conditions of Lemma \ref{lem_getsparse}. Then $v\leftrightarrow_{100\bar 1 1} v'$ with
$v'\in\{1,0,\overline{1}\}^*$ and the factor $020$, resp. $0\bar20$, which is at the right most position $i$, is rewritten into $011, 110$ or $010$,
($0\bar1\bar1, \bar1\bar10$ or $0\bar10$ respectively).
Moreover the digits of $v$ and $v'$ are equal for all indices $<i-1$.
\end{claim}

We prove this claim by induction on the number $N$ of appearances of the digits $\pm 2$. Of course the case $N=0$ is trivial.
For each $i$, with $v_i=\pm 2$ we define $\Delta_i=+\infty$ if it is the left most occurrence of the digits $\pm 2$ in $v$ and
$$
\Delta_i=\min\{j>0 : v_{i+j}=\pm 2\}
$$
otherwise. Without loss of generality, assume that $v_i=2$.
First, consider $\Delta_i\geq 6$. In this case we derive from (i)--(v) that we
only have four cases which can be rewritten as indicated below.
\begin{equation}\label{eq:Delta6}\begin{array}{c|c}
v_{i+4}v_{i+3}v_{i+2}|020| & \text{output} \\ \hline
00\overline{1}|020| & \overline{1}0\overline{1}|110| \\
010|020| & \overline{1}10|110| \\
000|020| & \overline{1}00|110| \\
\overline{1}00|020| & \overline{1}10|011|
\end{array}\end{equation}
Note that this also settles the case that $N=1$.

Now let us assume that $N\geq 2$ and that the claim is true for all words with strictly less than $N$ appearances of the digits $\pm 2$. Let us assume that $i$
is
the lowest index such that $|v_i|=2$.

If $\Delta_i\geq 6$, we split $v=u_2u_1$ into the two words $u_2=v_k\dots v_{i+6}v_{i+5}$ and $u_1=v_{i+4}v_{i+3}\dots v_\ell$.
Since $u_2$ has $N-1$ appearances of the digits $\pm 2$ we have by induction $v=u_2u_1\leftrightarrow u_2'u_1$ with
$u_2' \in \{\bar 1,0,1\}^*$. Now applying \eqref{eq:Delta6} we obtain $u_1\leftrightarrow u_1'$ with $u_1' \in \{\bar 1,0,1\}^*$, hence
$v=u_2u_1\leftrightarrow u_2'u_1'=v'$ with $v'\in\{\bar 1,0,1\}^*$.

If $\Delta_i=5$ we split up $v=u_2u_1$ into the two words $u_2=v_k\dots v_{i+5}v_{i+4}$ and $u_1=v_{i+3}v_{i+2}\dots v_\ell$. By induction
$u_2 \leftrightarrow u_2'=v_k'\dots v_{i+4}'$ and $u_2'\in\{\bar 1,0,1\}^*$ is a word ending with $\pm(011), \pm (110)$ or $\pm (010)$. Now the following
computations settle the case:
\begin{equation}\label{eq:Delta5}
\begin{array}{r|r|r}
(v_{i+6})v_{i+5}v_{i+4}|v_{i+3}v_{i+2}|020| & \text{rewriting $u_2$} & \text{final output}\\ \hline
20|00|020|\ & 10|00|020|\ & 1\overline{1}|00|110|\ \\
\ & 11|00|020|\ & 10|00|110|\ \\
\overline{2}0|00|020|\ & \overline{1}0|00|020|\ & \overline{11}|00|110|\ \\
\ & \overline{11}|00|020|\ & \overline{11}|10|011|\ \\
\qquad (0)\overline{2}0|10|020|\ & \overline{1}0|10|020| & \overline{11}|10|110|\ \\
\ & (0)\overline{11}|10|020|\ & (1)\overline{11}|11|011|\
\end{array}
\end{equation}

The case that $\Delta_i=4$ runs analogously. We split up $v=u_2u_1$ into two words $u_2=v_k\cdots v_{i+4}v_{i+3}$ and $u_1=v_{i+2}v_{i+1}\cdots v_\ell$ and
compute
\begin{equation}\label{eq:Delta4}
\begin{array}{c|c|c}
v_{i+4}v_{i+3}|v_{i+2}|020| & \text{rewriting $u_2$} & \text{final output} \\ \hline
\overline{2}0|0|020| & \overline{1}0|0|020| & \overline{1}1|0|011| \\
 & \overline{11}|0|020| & \overline{1}0|0|011|
\end{array}
\end{equation}

Now let us examine the case that $\Delta_i=3$.
We split up $v=u_2u_1$ into two words $u_2=v_k\dots v_{i+3}v_{i+2}$ and
$u_1=v_{i+1}v_{i}\dots v_\ell$ and get
\[\begin{array}{c|c|c}
v_{i+4}v_{i+3}v_{i+2}|020| & \text{rewriting $u_2$} & \text{final output}\\ \hline
020|020| & 011|020| &  \overline{1}11|110| \\
& 110|020| & 010|110| \\
& 010|020| & \overline{1}10|110|
\end{array}\]

We are left with the case $\Delta_i=2$.
In this case we split up $v=u_2u_1$ into $u_2=v_k\cdots v_{i+2}v_{i+1}$ and
$u_1=v_{i}v_{i-1}\cdots v_\ell$. Further, this case implies that $\Delta_{i+2}\geq 4$ and by (iv) also $v_{i+4}=v_{i+5}=0$, i.e. $u_2$ ends in $000\bar20$.
Looking at the possible rewritings of such words from
\eqref{eq:Delta6}, \eqref{eq:Delta5} and \eqref{eq:Delta4} we obtain the following cases:
\[\begin{array}{c|c|c}
v_{i+4}v_{i+3}v_{i+2}0|20| & \text{rewriting $u_2$} & \text{final output} \\ \hline
00\overline{2}0|20| & 00\overline{1}\overline{1}|20| & \overline{1}0\overline{1}0|10| \\
 & 0\overline{11}0|20| & \overline{111}1|10|
\end{array}\]
Therefore the proof of the claim and hence the proof of Proposition \ref{thm_combinatorial_dugV2} is complete.
\end{proof}

\begin{remark}
The method used in the proof of Proposition \ref{thm_combinatorial_dug} is very particular for the field $K=\Q(\gamma)$, where $\gamma$ is a root of the
polynomial $X^4-X+1$, which provided us rewriting rules $w$ with low weight but large support. We failed in proving an analogous result to Lemma 
\ref{lem_getsparse} for the remaining cases of Theorem \ref{Th:main}, since
the corresponding fields seem not to provide such rewriting rules.

As was mentioned in Section~\ref{Sec:context}, the possibility to rewrite any finite word with integer digits into the alphabet $\Sigma=\{-1,0,1\}$ is closely 
connected to the finiteness property~\eqref{eq_finproperty} of numeration systems. Being in general a highly nontrivial problem, only few results are known. 
For example, in~\cite{Frougny-Pelantova-Svobodova:2011}, it was shown that for any algebraic integer $\gamma$ without conjugates on the unit circle there exists 
an alphabet $\Sigma$ of consecutive integers, such that ${\rm Fin}_{\Sigma} (\gamma)$ is closed under addition. This is however very far from stating that 
$\Sigma=\{-1,0,1\}$ is sufficient.
\end{remark}

\section*{Acknowledgement}

This work was supported by the Czech Science Foundation, grant GA\v CR 13-03538S and by the Grant Agency of the Czech Technical University in Prague, grants No.
SGS11/162/OHK4/3T/14 and SGS14/205/OHK4/3T/14. The third author was supported by the Austrian Science Fund (FWF) under the project P~24801-N26.


\def\cprime{$'$}

\end{document}